\documentclass[a4paper,11pt]{article}

\usepackage{amsmath}
\usepackage{amssymb}
\usepackage{amsthm}
\usepackage{graphicx}
\usepackage{amscd}
\usepackage{epic, eepic}
\usepackage{url}
\usepackage[dvipsnames]{xcolor}
\usepackage[utf8]{inputenc} 
\usepackage{comment}
\usepackage{enumerate}   
\usepackage{graphicx}
\usepackage{epstopdf}
\usepackage{enumitem}
\usepackage{todonotes}
\usepackage{soul}

\usepackage[top=25mm, bottom=25mm, left=25mm, right=25mm]{geometry}
\usepackage[bookmarks=false,hidelinks]{hyperref}
\usepackage[capitalize]{cleveref}
\usepackage[normalem]{ulem}
\newcommand{\stkout}[1]{\ifmmode\text{\sout{\ensuremath{#1}}}\else\sout{#1}\fi}
\usepackage{tikz}
\usepackage[allcommands]{overarrows}
\usetikzlibrary{math,calc,intersections}
\usetikzlibrary{positioning,arrows,shapes,decorations.markings,decorations.pathreplacing, decorations.pathmorphing,matrix,patterns}
\tikzstyle{vertex}=[circle,draw=black,fill=black,inner sep=0,minimum size=5pt,text=white,font=\footnotesize]

\makeatletter
\renewenvironment{proof}[1][\proofname] {\par\pushQED{\qed}\normalfont\topsep6\p@\@plus6\p@\relax\trivlist\item[\hskip\labelsep\textit{#1}\@addpunct{.}]\ignorespaces}{\popQED\endtrivlist\@endpefalse}
\makeatother

\newtheorem{theorem}{\bf Theorem}[section]
\newtheorem{lemma}[theorem]{\bf Lemma}
\newtheorem{claim}[theorem]{\bf Claim}

\theoremstyle{definition}

\newcommand\claimproofend{\renewcommand{\qedsymbol}{$\boxdot$}
\end{proof}
\renewcommand{\qedsymbol}{$\square$}}

\def\eps{\varepsilon}

\def\Pb{\mathbb{P}}

\def\cB{\mathcal{B}}

\title{\vspace{-0.9cm}Cyclic subsets of tournaments}
\author{
Zach Hunter\thanks{Department of Mathematics, ETH Z\"urich, Switzerland. Email: {\tt \{zach.hunter, aleksa.milojevic, benjamin.sudakov\}@math.ethz.ch}. Research supported in part by SNSF grant 200021-228014.}
\and Teng Liu
\and Aleksa Milojevi\'c\footnotemark[1] \and Benny Sudakov\footnotemark[1]
}
\date{}

\begin{document}

\maketitle

\begin{abstract}
Let $G$ be a Dirac graph, and let $S$ be a vertex subset of $G$, chosen uniformly at random. How likely is the induced subgraph $G[S]$ to be Hamiltonian? This question, proposed by Erd\H{o}s and Faudree in 1996, was recently resolved by Dragani\'c, Keevash and M\"uyesser, in the setting of graphs. In this paper, we study a similar question for tournaments -- if $T$ is a tournament of high minimum degree, how likely is it for a random induced subtournament of $T$ to be Hamiltonian? We prove an optimal bound on this probability, and extend the results to the regime where the subset is not sampled uniformly at random, but according to a $p$-biased measure.
\end{abstract}

\section{Introduction}

A Hamilton cycle is a cycle visiting every vertex of a graph, and a graph containing a Hamilton cycle is called Hamiltonian. The study of Hamilton cycles is a central theme in graph theory, optimization, and theoretical computer science. For instance, deciding whether a graph is Hamiltonian is one of the best-known instances of an NP-complete problem, suggesting that it is impossible to obtain an easy to check condition fully characterizing Hamiltonian graphs. Therefore, for many decades the focus has been on identifying simple conditions which imply Hamiltonicity. The best example of such a condition is the classical Dirac's theorem \cite{D}, which states that every $n$-vertex graph of minimum degree $\delta(G)\geq n/2$ contains a Hamilton cycle. 

The theme of Hamiltonicity is even richer in the setting of directed graphs, where the edges are equipped with directions and one asks for a \textit{directed} Hamilton cycle. For example, observe that even the usual complete graph on $n$ vertices becomes much more interesting when directions are assigned to its edges, transforming it into a tournament (which is not necessarily always Hamiltonian). In the setting of directed graphs, the analogue of Dirac's theorem has been proven by Ghouila-Houri \cite{GH}, who showed that every $n$-vertex directed graph $D$ in which each vertex has at least $n/2$ incoming and outgoing edges contains a directed Hamilton cycle. In tournaments, however, the Hamilton cycles are somewhat easier to find, since already $\lfloor (n+2)/4\rfloor$ incoming and outgoing edges at each vertex suffice to guarantee a directed Hamilton cycle.

A very fruitful direction of research in the past two decades has been the study of how \textit{robustly} the above conditions imply Hamiltonicity (for a broader overview of the topic of robustness of graph properties, see \cite{S}). For example, given an $n$-vertex graph of minimum degree $n/2$, Dirac's theorem guarantees at least one Hamilton cycle, but can one guarantee many such cycles? This question was studied by S\'ark\"ozy, Selkow and Szemer\'edi \cite{SSS}, and ultimately resolved by Cuckler and Kahn in \cite{CK}, who proved that Dirac graphs contain at least $(1/2-o(1))^n n!$ Hamilton cycles. The counts of Hamilton cycles in random graphs have also been studied, for example in \cite{FKL, GK}.

Alternatively, resilience-type questions ask whether a Dirac graph, or a random graph, remains Hamiltonian even if a certain number of edges is removed from the graph adversarially (see e.g. \cite{FK, LS, M1, NST, SV}, and also \cite{HSS, FNNPS, M2} for the discussion of directed graphs). Furthermore, Krivelevich, Lee and Sudakov \cite{KLS} studied how likely Dirac graphs are to remain Hamiltonian, if their edges are sampled at random, and they showed that for an $n$-vertex graph $G$ with $\delta(G)\geq n/2$ and a probability $p\geq C\log n/n$, if every edge in $G$ is kept with probability $p$, then the resulting graph is Hamiltonian with high probability. Note that this theorem is optimal, up to the value of the constant $C$, since when $p<\log n/n$, the resulting graph will contain vertices of degree at most 1 with high probability.

It is equally natural to sample vertices instead of edges, and ask what is the probability that a random vertex-subset of a graph $G$ is Hamiltonian, if each vertex is kept with probability $p$. A version of this question was asked already in 1996 by Erd\H{o}s and Faudree \cite{E}. Namely, they asked to show that any $(n+1)$-regular graph $G$ on $2n$ vertices contains at least $\Omega(2^{2n})$ cyclic vertex-subsets $S\subseteq V(G)$, where $S$ is said to be \textit{cyclic} if there exists a cycle in $G$ with the vertex-set $S$. Observe that an equivalent way to phrase the conjecture is to take a random vertex-subset of $G$, including each vertex with probability $1/2$, and ask to show that the subgraph induced on this set is Hamiltonian with probability $\Omega(1)$.

This conjecture is tight in several ways -- for example, if $G$ is assumed to be a $n$-regular graph instead of an $(n+1)$-regular graph, then the conjecture does not hold anymore, which can be seen by considering the complete bipartite graph $G=K_{n, n}$. In this graph, the only vertex-subsets which are cyclic are those which intersect both sides of the bipartition in an equal number of vertices, and there are only $O(2^{2n}/\sqrt{n})$ such vertex-subsets. For a similar reason, it is not sufficient to ask for $G$ to have minimum degree $n+1$. Indeed, if we consider $G$ to be a complete bipartite graph $K_{n, n}$ with an $n$-vertex star added to both parts, then every cyclic vertex subset $G$ intersects both sides of the bipartition in a number of vertices differing by at most $2$. Again, there are only $O(2^{2n}/\sqrt{n})$ such vertex-subsets, showing that the conjecture does not hold without the regularity assumption.

The conjecture of Erd\H{o}s and Faudree was recently resolved by Dragani\'c, Keevash and M\"uyesser \cite{DKM}, who showed that a uniformly random vertex subset of an $(n+1)$-regular $2n$-vertex graph induces a Hamiltonian graph with probability at least $1/2-o(1)$, which is tight (as can be seen by considering a complete bipartite graph $K_{n-1, n+1}$ with a two-factor added to the larger side). Here, and throughout, $o(1)$ denotes a quantity which goes to $0$ as $n$ goes to infinity, with all other parameters held fixed.

In this paper, we study a natural variant of the Erd\H{o}s-Faudree question for tournaments. As we will see, tournaments differ from graphs in several ways -- for instance, as we previously mentioned, if a tournament $T$ has minimum semidegree\footnote{The \textit{minimum semidegree} of a tournament $T$ is the largest number $\delta$ such that $d^+(v), d^-(v)\geq \delta$ for all vertices $v\in V(T)$, i.e. such that each vertex of $T$ has at least $\delta$ incoming and outgoing edges.} $\delta^0(T)\geq \lfloor (n+2)/4\rfloor$, then $T$ must contain a directed Hamilton cycle. We show that this assumption also suffices to conclude that with probability $1/2-o(1)$ a uniformly random vertex-subset of $T$ is Hamiltonian. Note that no regularity assumption is required, in contrast to the case of graphs.

\begin{theorem}\label{thm:special_1}
Let $T$ be an $n$-vertex tournament with $\delta^0(T)\geq \lfloor \frac{n+2}{4}\rfloor$. Then, for a uniformly random subset $S\subseteq V(T)$, we have $\Pb[T[S]\text{ is Hamiltonian}]\geq 1/2-o(1)$.
\end{theorem}

Theorem~\ref{thm:special_1} is tight, both in terms of the assumption on $\delta^0(T)$ and in terms of the bound on the probability of $T[S]$ being Hamiltonian. Indeed, if, say $n=4k+2$, and $\delta^0(T)=k< \lfloor \frac{n+2}{4}\rfloor$, consider the tournament $T$ on the vertex set $V(T)=A\cup B$, where $|A|=|B|=2k+1$, $T[A], T[B]$ are $k$-regular tournaments, and $a\to b$ for all $a\in A, b\in B$. Recall that a tournament is Hamiltonian if and only if it is strongly connected, i.e. any vertex can be reached from any other using a directed path. Since no edges point from $B$ to $A$, $T[S]$ is Hamiltonian only if $S\cap A=\varnothing$ or $S\cap B=\varnothing$, which is \textit{very} unlikely. Thus, we must assume that $\delta^0(T)$ is at least as large as the Hamiltonicity threshold in order to hope that $T[S]$ is Hamiltonian with constant probability.

On the other hand, the constant $1/2$ in the conclusion cannot be improved either (unless $n=4k+1$, as we will see later). To see this, consider a tournament $T$ on $n=4k+3$ vertices with the vertex set $V(T)=A\cup B\cup \{v\}$, where $|A|=|B|=2k+1$. Let $T[A], T[B]$ be regular tournaments, with all vertices having in-degree and out-degree $k$. If we direct all edges from $A$ to $B$, from $B$ to $v$ and from $v$ to $A$, it is not hard to see that the resulting tournament has minimum semidegree $k+1$, as needed. Moreover, if $S$ is a uniform random subset of $V(T)$ which intersects both $A$ and $B$, $T[S]$ is Hamiltonian if and only if $v\in S$. Hence, the probability $T[S]$ is Hamiltonian is at most $\frac{1}{2}+o(1)$. 

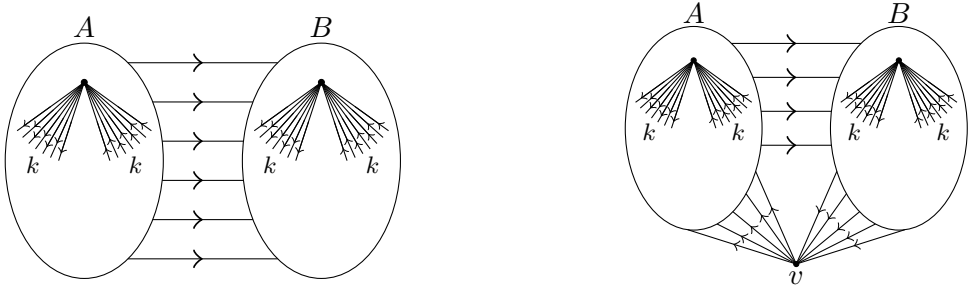
\begin{figure}[ht]
\begin{minipage}{0.48\textwidth}
\centering
\begin{tikzpicture}[scale=.52]

\coordinate (A) at (-3,0);
\coordinate (B) at (3,0);
\coordinate (U) at (-3,2);
\coordinate (V) at (3,2);

\draw (A) circle [y radius=3 cm,x radius=2cm];
\draw (B) circle [y radius=3 cm,x radius=2cm];

\draw ($(A)+(0,3.4)$) node {$A$};
\draw ($(B)+(0,3.4)$) node {$B$};

\foreach \ang/\delay in {-72/1.9,-66/1.7,-60/1.9,-54/1.7,-48/1.9,-42/1.7,-36/1.9}
{
\draw [black,-] ($(U)+(\ang:1.05)$) -- ($(U)+(\ang:2.10)$);
\draw [black,-<] (U) -- ($(U)+(\ang:\delay)$);
}

\foreach \ang/\delay in {-108/1.9,-114/1.7,-120/1.9,-126/1.7,-132/1.9,-138/1.7,-144/1.9}
{
\draw [black,-] ($(U)+(\ang:1.05)$) -- ($(U)+(\ang:2.10)$);
\draw [black,->] (U) -- ($(U)+(\ang:\delay)$);
}
\draw [fill] (U) [radius=0.075cm] circle ;
\draw ($(U)+(1.3,-2.1)$) node {\footnotesize$k$};
\draw ($(U)+(-1.3,-2.1)$) node {\footnotesize$k$};

\foreach \ang/\delay in {-72/1.9,-66/1.7,-60/1.9,-54/1.7,-48/1.9,-42/1.7,-36/1.9}
{
\draw [black,-] ($(V)+(\ang:1.05)$) -- ($(V)+(\ang:2.10)$);
\draw [black,-<] (V) -- ($(V)+(\ang:\delay)$);
}

\foreach \ang/\delay in {-108/1.9,-114/1.7,-120/1.9,-126/1.7,-132/1.9,-138/1.7,-144/1.9}
{
\draw [black,-] ($(V)+(\ang:1.05)$) -- ($(V)+(\ang:2.10)$);
\draw [black,->] (V) -- ($(V)+(\ang:\delay)$);
}
\draw [fill] (V) [radius=0.075cm] circle ;
\draw ($(V)+(1.3,-2.1)$) node {\footnotesize$k$};
\draw ($(V)+(-1.3,-2.1)$) node {\footnotesize$k$};

\foreach \h in {-2.5, -1.5, -0.5, 0.5, 1.5, 2.5} {
    \pgfmathsetmacro{\x}{2*sqrt(1 - \h*\h/9)-3}
    \pgfmathsetmacro{\xx}{-2*sqrt(1 - \h*\h/9)+3}
    \draw [black,
           postaction={decorate},
           decoration={markings, mark=at position 0.5 with {\arrow[scale=2]{>}}}]
          (\x, \h) -- (\xx, \h);
}

\end{tikzpicture}
\end{minipage}
\begin{minipage}{0.48\textwidth}
\centering
\begin{tikzpicture}[scale=.45]
\coordinate (A) at (-3,0);
\coordinate (B) at (3,0);
\coordinate (U) at (-3,2);
\coordinate (V) at (3,2);
\coordinate (X) at (0, -4);

\draw (A) circle [y radius=3 cm,x radius=2cm];
\draw (B) circle [y radius=3 cm,x radius=2cm];

\draw ($(A)+(0,3.4)$) node {$A$};
\draw ($(B)+(0,3.4)$) node {$B$};
\draw [fill] (X) [radius=0.075cm] circle ;
\draw ($(X)+(0,-0.4)$) node {$v$};

\foreach \ang/\delay in {-72/1.9,-66/1.7,-60/1.9,-54/1.7,-48/1.9,-42/1.7,-36/1.9}
{
\draw [black,-] ($(U)+(\ang:1.05)$) -- ($(U)+(\ang:2.10)$);
\draw [black,-<] (U) -- ($(U)+(\ang:\delay)$);
}

\foreach \ang/\delay in {-108/1.9,-114/1.7,-120/1.9,-126/1.7,-132/1.9,-138/1.7,-144/1.9}
{
\draw [black,-] ($(U)+(\ang:1.05)$) -- ($(U)+(\ang:2.10)$);
\draw [black,->] (U) -- ($(U)+(\ang:\delay)$);
}
\draw [fill] (U) [radius=0.075cm] circle ;
\draw ($(U)+(1.3,-2.1)$) node {\footnotesize$k$};
\draw ($(U)+(-1.3,-2.1)$) node {\footnotesize$k$};

\foreach \ang/\delay in {-72/1.9,-66/1.7,-60/1.9,-54/1.7,-48/1.9,-42/1.7,-36/1.9}
{
\draw [black,-] ($(V)+(\ang:1.05)$) -- ($(V)+(\ang:2.10)$);
\draw [black,-<] (V) -- ($(V)+(\ang:\delay)$);
}

\foreach \ang/\delay in {-108/1.9,-114/1.7,-120/1.9,-126/1.7,-132/1.9,-138/1.7,-144/1.9}
{
\draw [black,-] ($(V)+(\ang:1.05)$) -- ($(V)+(\ang:2.10)$);
\draw [black,->] (V) -- ($(V)+(\ang:\delay)$);
}
\draw [fill] (V) [radius=0.075cm] circle ;
\draw ($(V)+(1.3,-2.1)$) node {\footnotesize$k$};
\draw ($(V)+(-1.3,-2.1)$) node {\footnotesize$k$};

\foreach \h in {-0.5, 0.5, 1.5, 2.5} {
    \pgfmathsetmacro{\x}{2*sqrt(1 - \h*\h/9)-3}
    \pgfmathsetmacro{\xx}{-2*sqrt(1 - \h*\h/9)+3}
    \draw [black,
           postaction={decorate},
           decoration={markings, mark=at position 0.5 with {\arrow[scale=2]{>}}}]
          (\x, \h) -- (\xx, \h);
}

\foreach \ang/\delay in {-85/0.58, -110/0.7, -125/0.72, -140/0.7, -155/0.61} {
    \pgfmathsetmacro{\x}{2*cos(\ang)+3}
    \pgfmathsetmacro{\y}{3*sin(\ang)}
    \draw [black,
           postaction={decorate},
           decoration={markings, mark=at position \delay with {\arrow[scale=1.2]{<}}}]
          (X) -- (\x, \y);
}

\foreach \ang/\delay in {-85/0.58, -110/0.7, -125/0.72, -140/0.7, -155/0.61} {
    \pgfmathsetmacro{\x}{-2*cos(\ang)-3}
    \pgfmathsetmacro{\y}{3*sin(\ang)}
    \draw [black,
           postaction={decorate},
           decoration={markings, mark=at position \delay with {\arrow[scale=1.2]{>}}}]
          (X) -- (\x, \y);
}

\end{tikzpicture}\end{minipage}
    \caption{Examples showing tightness of Theorem~\ref{thm:special_1}.}
    \label{fig:enter-label}
\end{figure}
\medskip 

One of the advantages of working with tournaments is that the above result can be extended in several directions, say by considering different probability distributions over vertex-subsets or modifying the minimum degree conditions. For instance, Dragani\'c, Keevash and M\"uyesser ask whether an analogue of their results can be established if $S$ is a vertex-subset in which each vertex is included independently with probability $p\in (0, 1)$. Also, we can ask what happens if a stronger minimum semidegree condition is imposed onto $T$. Before we state the answer to these questions, let us introduce a piece of notation -- for a tournament $T$ and a probability $p\in (0, 1)$, we denote by $T_p$ the subtournament of $T$ induced on a random vertex subset $S\subseteq V(T)$ chosen by including each vertex randomly and independently with probability $p$. With this notation, we can state our main theorem, which answers the above questions and implies Theorem~\ref{thm:special_1} as a corollary by setting $t=1$. 

\begin{theorem}\label{thm:main}
Let $p\in (0, 1)$ and let $t\geq 1$ be an integer. For sufficiently large $n$ and an $n$-vertex tournament $T$ with $\delta^0(T)\geq \lfloor \frac{n-1-t}{4}\rfloor+t$, we have
\[\Pb[T_p\text{ is Hamiltonian}]\geq 1-(1-p)^{t}-o(1).\] 
If $n-t\equiv 1\bmod 4$, then the bound can be improved to $1-(1-p)^{t+1}-o(1)$.
\end{theorem}

Let us highlight two interesting differences between Theorem~\ref{thm:main} and the situation in undirected graphs. Firstly, observe that the above result obtains a characterization of the probability that $T_p$ is Hamiltonian for all $p\in (0, 1)$. Such a characterization is not available in the case of undirected graphs, where only the case $p=1/2$ is understood for far. Even further, Dragani\'c, Keevash and M\"uyesser propose to study what happens as $p$ varies and suggest that there may be a phase transition between two competing extremal examples -- something which we do not see in the world of tournaments.

On the other hand, as the required minimum degree is increased beyond the Hamiltonicity threshold, graphs and tournaments exhibit significantly different behaviour. Take for simplicity $p=1/2$, and consider an $(n+t)$-regular $2n$-vertex graph $G$ formed by adding a $2t$-regular graph into the larger part of $K_{n-t, n+t}$. If a random vertex subset $S$ intersects the part of size $n+t$ in fewer elements than the part of size $n-t$, then $G[S]$ is not Hamiltonian. Hence, if $t$ is a constant, $G[S]$ is Hamiltonian with probability only $1/2+o(1)$, which shows that increasing the regularity of the graph has no large effect on the probability that $G[S]$ is Hamiltonian. As Theorem~\ref{thm:main} shows, the situation is very different in tournaments. Indeed, if $p=1/2$ and $\delta^0(T)\geq \lfloor \frac{n-1-t}{4}\rfloor+t$, then a random induced subtournament $T[S]$ is Hamiltonian with probability $1-2^{-t}$.

\medskip

Let us conclude the introduction by presenting an example which shows that Theorem~\ref{thm:main} is tight. If $n-t\not\equiv 1\bmod 4$, let us consider the tournament $T$ with $V(T)=A\cup B\cup X$, where $|A|=\lfloor \frac{n-t}{2}\rfloor, |B|=\lceil \frac{n-t}{2}\rceil$ and $|X|=t$, with all edges directed from $A$ to $B$, from $B$ to $X$, and from $X$ to $A$. Moreover, inside $A$ and $B$, we put tournament with minimum semidegree at least $\lfloor \frac{|A|-1}{2}\rfloor, \lfloor \frac{|B|-1}{2}\rfloor$, and we direct the edges inside $X$ arbitrarily. It is then not hard to verify that the minimum semidegree is determined by the in-degree of the vertices in $A$, and those vertices have at least $\lfloor \frac{|A|-1}{2}\rfloor+t$ incoming edges. Hence, $\delta^0(T)\geq \lfloor \frac{|A|-1}{2}\rfloor\geq \lfloor\frac{n-t-2}{4}\rfloor+t$. However, since $n-t\not\equiv 1\bmod 4$, we have that $\delta^0(T)\geq \lfloor\frac{n-t-1}{4}\rfloor+t$, showing that $T$ satisfies the necessary assumptions.

If $S\subseteq V(T)$ is a random subset, including vertices with probability $p$, and we have $S\cap A, S\cap B\neq \varnothing$, then the tournament $T_p=T[S]$ is Hamiltonian precisely when $S$ contains a vertex of $X$. Thus, we have that \[\Pb[T_p\text{ is Hamiltonian}]\leq \Pb[V(T_p)\cap X\neq \varnothing]+o(1)=1-(1-p)^{t}+o(1).\]

{
}

\section{Proof overview}\label{sec:overview}

In this section, we present a short overview of the proof, postponing the details to later sections. To prove Theorem~\ref{thm:main}, we consider two cases, depending on whether there is a balanced bipartition $V(T)=A\cup B$ such that almost all edges go from $A$ to $B$ or not. More precisely, we choose a small parameter $\eps\in (0, 1)$ and we ask whether there is a balanced bipartition $V(T)=A\cup B$ for which the number of edges from $A$ to $B$ is at least $e(A, B)\geq (1-\eps)|A||B|$. We call such a bipartition an \textit{almost-directed cut}. If there is no such bipartition, the subtournament $T_p$ is Hamiltonian with probability $1-o(1)$, as shown by the following lemma.

\begin{lemma}\label{lemma:bidense case}
For every $p, \eps\in (0, 1)$ and $t\geq 1$, and for a large integer $n$, we have the following. Let $T$ be an $n$-vertex tournament satisfying $\delta^0(T)\geq \lfloor \frac{n-1-t}{4}\rfloor+t$, and suppose $T$ does not have a balanced bipartition $V(T)=A\cup B$ with $e(A, B)\geq (1-\eps)|A||B|$. Then
\[\Pb[T_p\text{ is Hamiltonian}]\geq 1-o(1).\]
\end{lemma}

Hence, we can focus on the case when $T$ contains an almost-directed cut $(A, B)$. In this case, we would aim to show that $T_p$ is Hamiltonian by showing that it is strongly connected. By a simple edge-counting argument, the tournaments $T[A], T[B]$ are almost regular, and thus very strongly connected (and this carries over to $T_p$). However, since most edges go from $A$ to $B$, the main issue in showing that $T_p$ is strongly connected will be obtaining paths from $B$ to $A$. Thus, our goal will be to identify many edges (or short paths) which go from $B$ to $A$ and use them to show that $T_p$ is strongly connected.

To do this, we need the following two definitions. Firstly, we call a partition $V(T)=A\cup B\cup X$ \textit{$\eps$-good} if $|A|, |B|\geq (1-\eps)n/2$, the minimum semidegree of induced subtournaments on $A, B$ is large $\delta^0(T[A]), \delta^0(T[B])\geq (1/6-\eps)n$, and $e(A, B)\geq (1-\eps) |A||B|$. Further, given such a partition, a vertex $v\in V(T)$ is a \textit{$k$-connector} if $|N^+(v)\cap A|\geq k$ and $|N^-(v)\cap B|\geq k$, where $k$ is a large constant we will specify later. The following two lemmas are crucial for resolving this case.

The first lemma shows that a bipartition $V(T)=A\cup B$ with $e(A, B)\geq (1-\eps) |A||B|$ can be converted into an $\eps$-good partition, which is shown by a simple cleaning procedure.

\begin{lemma}\label{lemma:good partitions exist}
Let $\eps>0$ be a small constant. If $T$ is a tournament of minimum semidegree $\delta^0(T)\geq \lfloor \frac{n+2}{4}\rfloor$ in which there exist a balanced bipartition $V(T)=A_0\cup B_0$ with $e(A_0, B_0)\geq (1-\eps)|A_0||B_0|$, then $T$ contains an $\eps^{1/3}$-good partition.
\end{lemma}

The next lemma shows that, as long as there are many paths from $B$ to $A$, either in the form of direct edges, or in the form of connectors, then the probability that $T_p$ is Hamiltonian is large.

\begin{lemma}\label{lemma:using connectors}
Let $p\in (0, 1), \eps\in (0, 10^{-2})$ and let $\sigma$ be a small constant. Then, for every $s\geq 1$, and every large enough $n$, we have the following two statements. 
\begin{enumerate}
    \item[(a)] If $T$ is a $n$-vertex tournament with minimum semidegree $\delta^0(T)\geq \lfloor \frac{n+2}{4}\rfloor$, with an $\eps$-good partition $V(T)=A\cup B\cup X$ and at least $s$ $k$-connectors, where $k\geq 2\log_{\frac{1}{1-p}}(s\sigma^{-1})$, then \[\Pb[T_p\text{ is Hamiltonian}]\geq 1-(1-p)^s-\sigma.\]
    \item[(b)] If $T$ is a $n$-vertex tournament with minimum semidegree $\delta^0(T)\geq \lfloor \frac{n+2}{4}\rfloor$, with an $\eps$-good partition $V(T)=A\cup B\cup X$ and a matching with at least $k$ edges directed from $B$ to $A$, then \[\Pb[T_p\text{ is Hamiltonian}]\geq 1-(1-p^2)^k-\sigma.\]
\end{enumerate}

\end{lemma}

With all of the above setup, in order to prove Theorem~\ref{thm:main}, it is suffices to show that $T$ contains at least $t$ $k$-connectors (or $t+1$ if $n-t\equiv 1\bmod 4$), and then apply Lemma~\ref{lemma:using connectors}. And indeed, in the final proof of Theorem~\ref{thm:main}, we show that if $T$ has an $\eps^{1/3}$-good partition, then either it contains $t$ connectors, or it contains a very large matching of edges from $B$ to $A$. This matching will be so large that with very good probability, one of the edges of the matching will be present in $T_p$, which will allow to ``go back'' from $B$ to $A$ and thus establish a Hamilton cycle in $T_p$. 

The rest of the paper is organized as follows: in Section~\ref{sec:bidense} we show that tournaments with no almost-directed cuts satisfy Theorem~\ref{thm:main} by giving a proof of Lemma~\ref{lemma:bidense case}. Then, in Section~\ref{sec:almost directed} we give proofs of Lemmas~\ref{lemma:good partitions exist} and~\ref{lemma:using connectors}. Finally, in Section~\ref{sec:completing} we complete the proof of Theorem~\ref{thm:main}.

\section{Tournaments with no almost-directed cuts}\label{sec:bidense}

The main goal of this section is to give a proof of Lemma~\ref{lemma:bidense case}, which we prepare with a sequence of lemmas. Before we do that, let us remark that Lemma~\ref{lemma:bidense case} can also be proven via the regularity method (see e.g. Section 3.1 of \cite{DKM} for a similar argument). However, we opt to give a more elementary, probabilistic proof of this lemma, which relies only on standard concentration inequalities.

We begin by showing that the degree statistics of the tournament $T$ do not change under sampling. More precisely, we show that with high probability, that for every vertex $v$ of in-degree $\alpha n$ in $T$, the in-degree of $v$ in $T_p$ is $(\alpha\pm o(1)) |V(T_p)|$. Then, we also show that if $T$ has $\alpha n$ vertices of in-degree at most $\beta n$, then we expect $T_p$ also has at most $(\alpha+o(1))|V(T_p)|$ vertices of in-degree at most $(\beta-o(1))|V(T_p)|$. 

\begin{claim}\label{claim:controlling the degrees 1}
With high probability, for each $v\in V(T_p)$ we have that \[\left|\frac{|N^-(v)\cap V(T_p)|}{|V(T_p)|}- \frac{d^-(v)}{n}\right|\leq o(1).\]
\end{claim}
\begin{proof}
Let $\sigma>0$ be a small constant, and let us prove that 
\[\Pb\left[\left|\frac{|N^-(v)\cap V(T_p)|}{|V(T_p)|}- \frac{d^-(v)}{n}\right|\leq 3\sigma\text{ for all }v\right]\leq 1-2(n+1)\exp\Big( - \frac{\sigma^2 p(n-1)}{12}\Big)=1-o(1).\]

The random variable $X=|N^-(v)\cap V(T_p)|$ is a sum of $d^-(v)$ Bernoulli indicator variables with mean $p$. Thus, by Chernoff bounds (see e.g. Section 2.1 of \cite{JLR}), we have 
\[\Pb\Big[\big|X -pd^-(v)\big|\geq \sigma pd^-(v)\Big]\leq 2\exp\left(-\frac{\sigma ^2\cdot pd^-(v)}{3}\right)\leq 2\exp\left(-\frac{\sigma^2p (n-1)}{12}\right),\]
where we have used that $d^-(v)\geq \lfloor \frac{n+2}{4}\rfloor\geq \frac{n-1}{4}$ in the last inequality. 

Similarly, the random variable $|V(T_p)|$ is a sum of $n$ Bernoulli indicator variables with mean $p$, and Chernoff bounds imply that \[\Pb\Big[\big||V(T_p)| -pn\big|\geq \sigma pn\Big]\leq 2\exp\left(-\frac{\sigma^2\cdot pn}{3}\right)\leq 2\exp\left(-\frac{\sigma^2p n}{3}\right),\]
Hence, by the union bound over all vertices $v$, together with the event that $|V(T_p)|$ is close to its expectation, we conclude that with probability $1-2(n+1)\exp\big( -{\sigma^2 p(n-1)}/{12}\big)=1-o(1)$, for each vertex $v$ we have that $\big||N^-(v)\cap V(T_p)|- p d^-(v)\big|\leq \sigma pd^-(v)$ and $\big||V(T_p)| -pn\big|\leq \sigma pn$. Then, for each $v$ we have
\begin{align*}
&\frac{|N^-(v)\cap V(T_p)|}{|V(T_p)|}- \frac{d^-(v)}{n}\leq \frac{(1+ \sigma) pd^-(v)}{(1- \sigma) pn}- \frac{d^-(v)}{n}= \frac{d^-(v)}{n}\left(\frac{1+ \sigma}{1-\sigma}- 1\right)\leq 3\sigma,\text{ and}\\
&\frac{|N^-(v)\cap V(T_p)|}{|V(T_p)|}- \frac{d^-(v)}{n}\geq \frac{(1- \sigma) pd^-(v)}{(1+ \sigma) pn}- \frac{d^-(v)}{n}= \frac{d^-(v)}{n}\left(\frac{1- \sigma}{1+\sigma}- 1\right)\geq -3\sigma.\qedhere
\end{align*}
\end{proof}

\begin{claim}\label{claim:controlling the degrees 2}
Let $\alpha, \beta, \delta \in (0, 1)$ be fixed real numbers and let $n$ be a large integer. Suppose that an $n$-vertex tournament $T$ contains at most $\alpha n$ vertices with $d^{-}(v)\leq \beta n$. Then, with probability at least $1-o(1)$, the tournament $T_p$ contains at most $(\alpha +o(1)) |V(T_p)|$ vertices with $|N^{-}(v)\cap V(T_p)|\leq (\beta -\delta) |V(T_p)|$. 
\end{claim}
\begin{proof}
Our goal is to bound the size of $X_p=\{v\in V(T_p):|N^{-}(v)\cap V(T_p)|\leq (\beta -\delta) |V(T_p)|\}$. Let $\sigma \ll \delta$ be a small positive constant -- we will prove that $\Pb\big[|X_p|\leq (\alpha+3\sigma)|V(T_p)|\big]\geq 1-\sigma$. 

Let us assume that the outcome of Claim~\ref{claim:controlling the degrees 1} holds, i.e. that for every vertex $v\in V(T_p)$ we have \begin{equation}\label{eq:1}
    \left|\frac{|N^-(v)\cap V(T_p)|}{|V(T_p)|}- \frac{d^-(v)}{n}\right|\leq \sigma.
\end{equation} By Claim~\ref{claim:controlling the degrees 1}, this happens with probability at least $1-\sigma/3$, if $n$ is large enough. Also, let us assume that $|V(T_p)|\geq (1-\sigma)pn$, which also happens with probability at least $1-\sigma/3$, due to the standard concentration inequalities (as in the proof of Claim~\ref{claim:controlling the degrees 1}). Finally, define $X=\{v\in V(T):d^-(v)\leq \beta n\}$ to be the set of low in-degree vertices, for which we know that $|X|\leq \alpha n$ by the assumption of Claim~\ref{claim:controlling the degrees 2}.

By (\ref{eq:1}) and $\sigma\ll \delta$, we know that each vertex of $X_p$ must also be in $X$, i.e. $X_p\subseteq X\cap V(T_p)$. So, we have $|X_p|\leq |X\cap V(T_p)|$, where we note that $|X\cap V(T_p)|$ is a sum of $|X|$ Bernoulli random variables with mean $p$. Hoeffding's inequality implies
\[\Pb\big[|X_p|\geq (\alpha+\sigma)p n \big]\leq \Pb\big[|X\cap V(T_p)|\geq p|X|+\sigma pn\big]\leq \exp\left(-\frac{2(\sigma p n)^2}{|X|} \right)\leq 2\exp\big(-2\sigma^2 p^2 n\big)\leq \sigma/3.\]
To complete the proof, observe that if $|X_p|\leq (\alpha+\sigma) pn$, then 
\[|X_p|\leq (\alpha+\sigma)p n\leq \frac{\alpha+\sigma}{1-\sigma}|V(T_p)|\leq (\alpha+3\sigma)|V(T_p)|.\]
Hence, if $n$ is large enough, with probability at least $1-\sigma$ we have $|X_p|\leq (\alpha+3\sigma)|V(T_p)|$.
\end{proof}

We need one more preparatory claim before the proof of Lemma~\ref{lemma:bidense case}. 

\begin{claim}\label{claim:stability}
Let $\delta\in (0, 1/2)$ be a constant and let $T$ be an $n$-vertex tournament with $\delta^0(T)\geq (1/4-\delta^2)n$ with no Hamilton cycle. Then, $T$ contains at least $(1/2-2\delta) n$ vertices of in-degree smaller than $(1/4+2\delta)n$. 
\end{claim}
\begin{proof}
If $T$ is not Hamiltonian, it is not strongly connected and so there exists a directed cut $(A, B)$. Since each vertex $v\in V(T)$ satisfies $d^+(v)\geq n/4-\delta^2 n$ and there are no edges going from $B$ to $A$, we have that 
\[\binom{|B|}{2}=\sum_{v\in B}d^+(v)\geq |B|\cdot (n/4-\delta^2 n),\]
which shows that $|B|\geq n/2-2\delta^2 n$. Thus, $|A|=|V(T)|-|B|\leq n-n/2+2\delta^2 n=n/2+2\delta^2 n$. 

Similarly, by computing the sum of in-degrees of the vertices in $A$ and using that $d^-(v)\geq n/4-\delta^2 n$ for all $v\in A$, we find that $\binom{|A|}{2}=\sum_{v\in A}d^-(A)\geq |A|\cdot (n/4-\delta^2 n)$, which shows that $|A|\geq n/2-2\delta^2 n\geq n/2-\delta n$ (since $\delta\leq 1/2$).

Let $r$ be the number of vertices $v\in A$ with $d^-(v)\geq (1/4+2\delta) n$. To finish the proof, we need to show that $r\leq \delta n$, since this implies that at least $|A|-\delta n\geq n/2-2\delta n$ vertices have $d^-(v)\leq (1/4+2\delta) n$. 

Let us consider again the sum of in-degrees of vertices $v\in A$, now taking into account the $r$ vertices of high in-degree. 
Since every vertex of $A$ has in-degree at least $\delta^0(T)\geq (1/4-\delta^2)n$, and there exist $r$ vertices which have additional $2\delta n$ incoming edges, we have
\[\binom{|A|}{2} = \sum_{v\in A}d^-(v)\geq |A| (n/4-\delta^2 n) + r\cdot 2\delta n.\]
Rearranging, we find that 
\[r\leq \frac{1}{2\delta n}\cdot |A|\left(\frac{|A|-1}{2}-\frac{n}{4}+\delta^2 n\right) \leq \frac{1}{2\delta n}\cdot (1/2+2\delta^2) n \cdot 2\delta^2 n  \leq \delta n,\]
where we have used that $|A|\leq n/2+2\delta^2 n$ in the second inequality. This shows $r\leq \delta n$ and thus completes the proof.
\end{proof}

\begin{proof}[Proof of Lemma~\ref{lemma:bidense case}.]
Set $\delta=\eps/20$ and set $\alpha=1/2-3\delta, \beta=1/4+3\delta$. We have two cases — either $T$ contains more than $\alpha n$ vertices with $d^-(v)\leq \beta n$, or not. If $T$ contains at least $\alpha n$ such vertices, we will show that there exists a balanced bipartition $V(T)=A\cup B$ such that $e(A, B)\geq (1-\eps) |A||B|$.

Namely, if we have at least $(1/2-3\delta)n$ vertices with $d^+(v)\geq n-1-d^-(v)\geq n-1-\beta n$, let $A$ be any set of size $n/2$ which contains more than $\alpha n$ such vertices, and let $B=V(T)\backslash A$. Then, we have
\[\binom{|A|}{2}+e(A, B)=\sum_{a\in A}d^+(a)> \alpha n \cdot (n-1-\beta n).\]
Thus, we have
\[e(A, B)\geq \Big(\frac{1}{2}-3\delta\Big)\Big(\frac{3}{4}-4\delta\Big)n^2-\binom{n/2}{2}\geq \frac{3}{8}n^2-5\delta n^2-\frac{n^2}{8}\geq (1-20\delta)\frac{n^2}{4}.\]
Since $20\delta\leq \eps$, we have a balanced bipartition $V(T)=A\cup B$ with $e(A, B)\geq (1-\eps) |A||B|$.

Thus, we may assume that we have at most $\alpha n$ vertices with $d^-(v)\leq \beta n$. By Claim~\ref{claim:controlling the degrees 1}, with probability $1-o(1)$, for each $v\in V(T_p)$ we have that $|N^-(v)\cap V(T_p)|\geq (1/4- \delta^2) n'$ and $|N^+(v)\cap V(T_p)|\geq (1/4- \delta^2) n'$, where $n'$ denotes the number of vertices of $T_p$. Hence, we also have $\delta^0(T_p)\geq (1/4-\delta^2) n'$.

Further, by Claim~\ref{claim:controlling the degrees 2}, with probability $1-o(1)$, applied with $\alpha, \beta$ and $\delta$, $T_p$ contains fewer than $(\alpha+o(1)) n'\leq (1/2-2\delta)n'$ vertices of in-degree smaller than $(\beta-\delta)n'= (1/4+2\delta)n'$. But then, by Claim~\ref{claim:stability}, we conclude that $T_p$ must be Hamiltonian. So, $T_p$ is Hamiltonian with probability at least $1-o(1)$, as needed.
\end{proof}

\section{Tournaments with almost-directed cuts}\label{sec:almost directed}

In this section, we prove Lemmas~\ref{lemma:good partitions exist} and~\ref{lemma:using connectors}. Recall, Lemma~\ref{lemma:good partitions exist} states that if an almost-directed cut exists in a tournament $T$, then there also exist $\eps^{1/3}$-good partitions.

\begin{proof}[Proof of Lemma~\ref{lemma:good partitions exist}.]
In this proof, we will identify sets of vertices of $T[A_0], T[B_0]$ which have too small in-degree or out-degree, and eliminate them in order to form sets $A$ with $\delta^0(T[A])\geq n/6$, $\delta^0(T[B])\geq n/6$. To do this, let $\delta=\eps^{1/2}$ and let $A_0^-=\{v\in A_0: |N^-(v)\cap A_0|\leq (1/4-\delta)n\}$ and $A_0^+=\{v\in A_0: |N^+(v)\cap A_0|\leq n/5\}$.

Let us show $|A_0^-|\leq \delta n/4$. Since $d^-(v)\geq \lfloor \frac{n+2}{4}\rfloor$ for all $v\in V(T)$, each vertex of $A_0^-$ has at least $\delta n$ incoming edges from $B_0$. Since $e(B_0, A_0)\leq \eps |A_0||B_0|=\eps n^2/4$, we find that $|A_0^-|\cdot \delta n\leq \eps n^2/4$, implying $|A_0^-|\leq \delta n/4$, since $\delta=\eps^{1/2}$. 

Next, we show that $|A_0^+|\leq 15\delta n$. Observe that vertices in $A_0^-$ have $|N^+(v)\cap A_0|\leq |A_0|\leq n/2$ and vertices outside $A_0^-$ have $|N^+(v)\cap A_0|\leq |A_0|-d^-(v)\leq (1/4+\delta)n$. Thus, 
\begin{align*}
\binom{|A_0|}{2}=\sum_{v\in A_0}|N^+(v)\cap A_0|&\leq |A_0^+|\cdot \frac{n}{5}+|A_0^-|\cdot \frac{n}{2}+(|A_0|-|A_0^+|-|A_0^-|)\cdot \Big(\frac{n}{4}+\delta n\Big)\\
&\leq |A_0|\cdot \Big(\frac{n}{4}+\delta n\Big) - |A_0^+|\cdot \frac{n}{20}+\frac{\delta n^2}{8},
\end{align*}
where we have used that $|A_0^-|\leq \delta n/4$. Rearranging, we find that 
\[|A_0^+|\cdot \frac{n}{20}\leq |A_0|\cdot \Big(\frac{n}{4}+\delta n-\frac{|A_0|-1}{2}\Big) + \frac{\delta n^2}{8}\leq \frac{3}{4}\delta n^2,\]
where we used that $|A_0|=n/2$. We conclude that $|A_0^+|\leq 15\delta n$. 

This shows that $|A_0^+\cup A_0^-|\leq 16 \delta n$. Thus, if $A=A_0\backslash (A_0^-\cup A_0^+)$, then every vertex $v\in A$ has $|N^+(v)\cap A_0|, |N^-(v)\cap A_0|\geq n/5$. Moreover, since at most $16\delta n$ neighbours of $v$ are in $A_0\backslash A$, we see that $|N^+(v)\cap A|, |N^-(v)\cap A|\geq n/5-16\delta n\geq n/6$, as long as $\delta$ is small enough.

Similarly, we can define $B_0^+=\{v\in B_0: |N^+(v)\cap B_0|\leq (1/4-\delta)n\}$ and $B_0^-=\{v\in B_0: |N^-(v)\cap B_0|\leq n/5\}$. In a completely analogous way as above, we can obtain bound $|B_0^+|\leq \delta n/4$ and $|B_0^-|\leq 15\delta n$, ultimately concluding that the set $B=B_0\backslash (B_0^-\cup B_0^+)$ induces a subtournament of $T$ with minimum degree $n/6$.

If $X=V(T)\backslash (A\cup B)$, then we claim that $A, B$ and $X$ form an $\eps^{1/3}$-good partition. We have already verified that $\delta^0(T[A]), \delta^0(T[B])\geq n/6$. Also, since $\delta=\eps^{1/2}$ and $\eps$ is sufficiently small, we can easily lower bound the size of $A$ and $B$ so that $|A|, |B|\geq \frac{n}{2}-16\eps^{1/2}n\geq (1-\eps^{1/3})n/2$. Finally, there are at most $2\cdot 16 \eps^{1/2} n^2$ edges touching the vertices removed from $A_0$ and $B_0$, and thus $e(A, B)\geq e(A_0, B_0)-32\eps^{1/2}n^2\geq (1-\eps^{1/3})|A||B|$. This verifies that $V(T)=A\cup B\cup X$ is indeed an $\eps^{1/3}$-good partition and completes the proof. 
\end{proof}

Let us now turn to the proof of Lemma~\ref{lemma:using connectors}. To prepare its proof, we will identify a collection of bad events which prevent $T_p$ from being Hamiltonian. As long as these bad events do not occur, we will deterministically show that $T_p$ is Hamiltonian.

If $T$ is a tournament with minimum degree at least $\lfloor \frac{n+2}{4}\rfloor$ and an $\eps$-good partition $V(T)=A\cup B\cup X$, and $S$ is a random vertex subset, then we define the bad events as follows:
\begin{itemize}
    \item let $\cB_1$ be the event that $|S\cap X|\geq |S|/5$,
    \item let $\cB_2$ be the event that $\delta^0(T[A\cap S])< \frac{3}{10}|A\cap S|$, $\delta^0(T[B\cap S])< \frac{3}{10}|B\cap S|$, or $\delta^0(T[S])< |S|/5$,
    \item let $\cB_3$ be the event that in $T_p$ there is no directed path from $A\cap S$ to $B\cap S$, and
    \item let $\cB_4$ be the event that in $T_p$ there is no directed path from $B\cap S$ to $A\cap S$.
\end{itemize}

The following claim shows that, as long as the events $\cB_1, \cB_2$, $\cB_3$ and $\cB_4$ are avoided, we have that $T_p=T[S]$ is Hamiltonian.

\begin{claim}\label{claim:proving hamiltonicity}
Let $p\in (0, 1)$, $\eps\in (0, 10^{-2})$, and let $n$ be a sufficiently large integer. If $T$ is a tournament with an $\eps$-good partition $V(T)=A\cup B\cup X$, and if the events $\cB_1, \cB_2, \cB_3, \cB_4$ do not hold, then $T_p$ is Hamiltonian.
\end{claim}
\begin{proof}
We will show that $T_p$ is Hamiltonian by showing it is strongly connected, meaning that between any two vertices $u, w\in S$ there is a directed walk from $u$ to $w$ in $T_p$. Let us consider several cases, based on which part of the partition the vertices $u, w$ belong to. 

As a first step, observe that from any $u\in X$, one can take an out-edge to a vertex in $A\cup B$, since the minimum semidegree of $T_p$ is at least $\delta^0(T[S])\geq |S|/5 > |S\cap X|$ (where we use that the events $\cB_1$ and $\cB_2$) do not hold. For the same reason, every $w\in X$ has an incoming edge from $A\cup B$. This observation shows that if suffices to verify that there is a directed path between any pair of vertices in $A\cup B$, since this is enough to guarantee the strong connectivity of the whole $T_p$ by using the incoming/outgoing edges from $X$.

Secondly, observe that $\delta^0(T[A\cap S])\geq \frac{3}{10}|A\cap S|> \lfloor \frac{|A\cap S|-2}{4}\rfloor$, implying that $T[A\cap S]$ is Hamiltonian and thus strongly connected. Thus, if $u, w\in A\cap S$, there exists the directed walk from $u$ to $w$. Similarly, $T[B\cap S]$ is also strongly connected, showing that if $u, w\in B\cap S$, then we also have a directed walk from $u$ to $w$. 

Finally, let us consider the case when $u, w$ are in different parts, say $u\in A\cap S$ and $w\in B\cap S$. Since $\cB_3$ does not hold, there is a directed path from some $a\in A\cap S$ to some $b\in B\cap S$. Furthermore, since $T[A\cap S], T[B\cap S]$ are strongly connected, there is a directed walk from $u$ to $a$ and from $b$ to $w$. Concatenating these walks and the path from $a$ to $b$, we conclude that there is a directed walk from $u$ to $w$, completing this case as well. Note that the last case, when $u\in B\cap S$ and $w\in A\cap S$, is symmetric, with the roles of $A$ and $B$ reversed. Thus, the analysis we have given so far completes the proof.
\end{proof}

We are now ready to prove Lemma~\ref{lemma:using connectors} using Claim~\ref{claim:proving hamiltonicity}.

\begin{proof}[Proof of Lemma~\ref{lemma:using connectors}.]
Let $V(T)=A\cup B\cup X$ be the $\eps$-good partition of $T$, and let us denote by $S$ the vertex set of $T_p$. By Claim~\ref{claim:proving hamiltonicity}, we know that if $T_p$ is not Hamiltonian, then one of the events $\cB_1, \cB_2, \cB_3$ or $\cB_4$ must hold. Thus, let us begin by bounding the probabilities of $\cB_1, \cB_2$ and $\cB_3$.

We start by showing $\Pb[\cB_1]\leq \sigma/4$. Note that $|S\cap X|$ is a binomial random variable with mean $p|X|\leq p\eps n$, since $|X|= n-|A|-|B|\leq n-2\frac{1-\eps}{2}n=\eps n$. Thus, by Chernoff bounds, we have that 
\[\Pb[|S\cap X|\geq 2p\eps n]\leq \exp\Big(-\frac{\eps p n}{3}\Big)\leq \sigma/8,\]
as long as $n$ is large enough. Also, the probability that $|S|<pn/2$ is at most $\exp\Big(-\frac{p n/2}{3}\Big)\leq \sigma/8$, since the expectation of $|S|$ is $pn$. We have $\eps\leq \frac{1}{20}$, and therefore $|S\cap X|\leq 2\eps pn$ and $pn/2\leq |S|$ implies $|S\cap X|\leq |S|/5$. Hence, $\Pb[\cB_1]\leq \sigma/4$.

Next, we show that $\Pb[\cB_2]\leq \sigma/4$. By Claim~\ref{claim:controlling the degrees 1} applied to $T[A]$, with probability at least $1-\sigma/12$, for all $v\in V(T_p)$ we have that 
\begin{equation}\label{eq:2}
    \left|\frac{|N^-(v)\cap A\cap S|}{|A\cap S|}- \frac{d_{T[A]}^-(v)}{|A|}\right|,\left|\frac{|N^+(v)\cap A\cap S|}{|A\cap S|}- \frac{d_{T[A]}^+(v)}{|A|}\right|\leq \frac{1}{100}.
\end{equation}
Since $V(T)=A\cup B\cup X$ is an $\eps$-good partition, we have $d_{T[A]}^-(v), d_{T[A]}^+(v)\geq (\frac{1}{6}-\eps)n\geq \frac{31|A|}{100}$, and so (\ref{eq:2}) implies that inside $A\cap S$ all vertices have degree at least $\frac{3}{10}|A\cap S|$. Thus, the probability that $\delta^0(T[A\cap S])<\frac{3}{10}|A\cap S|$ is at most $\sigma/12$. An analogous argument shows that the probability that $\delta^0(T[B\cap S])<\frac{3}{10}|B\cap S|$ is at most $\sigma/12$. Finally, by applying Claim~\ref{claim:controlling the degrees 1} to the tournament $T$ we find that with probability at least $1-\sigma/12$ we have
\begin{equation*}
    \left|\frac{|N^-(v)\cap S|}{|S|}- \frac{d_{T}^-(v)}{n}\right|,\left|\frac{|N^+(v)\cap S|}{|S|}- \frac{d_{T}^+(v)}{n}\right|\leq \frac{1}{100}.
\end{equation*}
Recalling that $d^+_T(v), d^-_T(v)\geq n/4$ for all $v\in V(T)$, we can easily deduce from the above inequality that $|N^+(v)\cap S|, |N^-(v)\cap S|\geq (\frac{1}{4}-\frac{1}{100})|S|\geq \frac{1}{5}|S|$ for all $v\in S$. By a union bound, this shows that $\Pb[\cB_2]\leq 3\cdot \sigma/12\leq \sigma/4$.

Finally, let us argue that $\Pb[\cB_3]\leq \sigma/4$. In fact, we bound the probability no edge $a\to b$ with $a\in A, b\in B$ has both $a, b\in S$, which is clearly an upper bound on $\Pb[\cB_3]$. Observe that for at least half the vertices of $A$, we have $|N^+(a)\cap B|\geq |B|/2$ (if this was not the case, would could not have more than $\frac{3}{4}|A||B|$ edges from $A$ to $B$). Now, the probability that $S$ contains no vertex $a\in A$ with $|N^+(a)\cap B|\geq |B|/2$ is at most $(1-p)^{|A|/2}$. If $S$ contains such a vertex, the probability that it contains no element of $N^+(a)\cap B$ is at most $(1-p)^{|B|/2}$. Thus, $\Pb[\cB_3]\leq (1-p)^{|A|/2}+(1-p)^{|B|/2}\leq 2(1-p)^{n/3}\leq \sigma/4$.

Let us now split the proof into two parts, and first focus on the part \textit{(a)} of the Lemma~\ref{lemma:using connectors}.
\medskip

\noindent
\textbf{Part {(a)}.}
Let $C$ be the set of $s$ $k$-connectors in $T$, which exists by assumption. We have $\Pb[C\cap S\neq \varnothing]=1-(1-p)^s$ and therefore, in order to prove that $\Pb[T_p\text{ is Hamiltonian}]\geq 1-(1-p)^s-\sigma$, it suffices to show $\Pb[C\cap S\neq \varnothing\text{ and }T_p\text{ is not Hamiltonian}]\leq \sigma$. Recall that by Claim~\ref{claim:proving hamiltonicity}, we know that if $T_p$ is not Hamiltonian, then one of the events $\cB_1, \cB_2, \cB_3$ or $\cB_4$ must hold. Hence,
\begin{align*}
\Pb\big[C\cap S\neq \varnothing\text{ and }T_p\text{ is not Hamiltonian}\big]
&\leq \Pb[C\cap S\neq \varnothing\text{ and }(\cB_1\vee \cB_2\vee \cB_3\vee \cB_4)]\\
&\leq \Pb[\cB_1]+ \Pb[\cB_2]+ \Pb[\cB_3]+ \Pb[C\cap S\neq \varnothing\text{ and }\cB_4]\\
&\leq \Pb[\cB_1]+ \Pb[\cB_2]+ \Pb[\cB_3]+ \sum_{v\in C} \Pb[v\in S\text{ and }\cB_4 ].
\end{align*}
Note that we have used a union bound over all $v\in C$ in the last inequality. We know that the first three probabilities are at most $\sigma/4$, and therefore it suffices to show that each of the probabilities in the last sum $\sigma/4s$. To do this, we will fix a vertex $v\in C$ and show that $\Pb\big[v\in S\text{ and } \cB_4\big]\leq \sigma/4s$. In fact, it will be easier to show that $\Pb[\cB_4|v\in S]\leq \sigma/4s$, which is still sufficient. 

The probability that $N^-(v)\cap B\cap S=\varnothing$ is at most $(1-p)^{k}$, since $|N^-(v)\cap B|\geq k$ due to the fact $v$ is a $k$-connector. Similarly, $\Pb[N^+(v)\cap A\cap S=\varnothing]\leq (1-p)^k$, and therefore $v$ has an in-neighbour $b\in B\cap S$ and an out-neighbour $a\in A\cap S$ with probability at least $1-2(1-p)^k\geq 1-2(\sigma/s)^2\geq 1-\sigma/4 s$ (where we have used that $k\geq 2\log_{\frac{1}{1-p}}(s\sigma^{-1})$). Since we are given that $v\in S$, we know that with probability $1-\sigma/4s$ there is a directed path from $B$ to $A$ in $T_p$. This shows that $\Pb[\cB_4|v\in S]\leq \sigma/4s$, thus completing the proof of the statement \textit{(a)}.
\medskip

\noindent
\textbf{Part {(b)}.} Let $M$ be a matching of the edges directed from $B$ to $A$ of size $k$. If there is an edge $e\in M$ with $e\subset S$, then $\cB_4$ does not hold. Thus, $\Pb[\cB_4]\leq \Pb[\text{ no $e\in M$ has $e\subset S$}]= (1-p^2)^k$, where we have used the fact that each $e\in M$ has probability $p^2$ of being fully included in $S$, and these events are independent for all edges since $M$ is a matching. Thus, 
\[\Pb[T_p\text{ is Hamiltonian}]\geq 1-\Pb[\cB_1]- \Pb[\cB_2]- \Pb[\cB_3]-\Pb[\cB_4]\geq 1-(1-p^2)^k-3/4\sigma,\]
which is sufficient to complete the proof.
\end{proof}

\section{Completing the proof}\label{sec:completing}

\begin{proof}[Proof of Theorem~\ref{thm:main}.]

Let us fix a small constant $\sigma>0$, and let us show that for sufficiently large $n$ we have $\Pb[T_p\text{ is Hamiltonian}]\geq 1-(1-p)^t-\sigma$ (or $\Pb[T_p\text{ is Hamiltonian}]\geq 1-(1-p)^{t+1}-\sigma$ if $n-t\equiv 1\bmod 4$). Let us also fix parameters $k=2\log_{\frac{1}{1-p^2}}((t+1)\sigma^{-1})>2\log_{\frac{1}{1-p}}((t+1)\sigma^{-1})$ and $\eps=\frac{1}{100k}$. Note that the stated inequality for $k$ holds since $\frac{1}{1-p^2}<\frac{1}{1-p}$ and we need $k>2\log_{\frac{1}{1-p}}((t+1)\sigma^{-1})$ in order to apply the part \textit{(a)} of Lemma~\ref{lemma:using connectors}.

By Lemma~\ref{lemma:bidense case}, we may assume that there exists a balanced bipartition $V(T)=A\cup B$ for which $e(A, B)\geq (1-\eps^3) |A||B|$, since we otherwise have $\Pb[T_p\text{ is Hamiltonian}]\geq 1-\sigma$. Thus, by applying Lemma~\ref{lemma:good partitions exist}, we find that there exists an $\eps$-good partition $V(T)=A_0\cup B_0\cup X_0$. Let us pick an $\eps$-good partition $V(T)=A_0\cup B_0\cup X_0$ which minimizes $|X_0|$. 

Finally, by part \textit{(a)} of Lemma~\ref{lemma:using connectors}, we are immediately done if there are at least $t+1$ vertices which are $k$-connectors. Thus, we may assume there are fewer than $t+1$ vertices $b\in B_0$ which have $|N^+(b)\cap A_0|\geq k+t$, and similarly fewer than $t+1$ vertices $a\in A_0$ which have $|N^-(a)\cap B_0|\geq k+t$ (because all such vertices are $k$-connectors). Let us now move all such vertices $A_0, B_0$ to $X_0$, noting that we move at most $t$ vertices. Let us denote the resulting partition by $V(T)=A\cup B\cup X$. 

Each vertex $b\in B_0$ which was moved to $X$ still satisfies $|N^+(v)\cap A|\geq (k+t)-t=k$, and hence remains a $k$-connector (since it is also adjacent to at least $(1/6-\eps)n-t$ vertices of $B$). Similarly, each vertex $a\in A_0$ which was moved to $X$ remains a $k$-connector. Also, the partition $V(T)=A\cup B\cup X$ still satisfies that $e(A, B)\geq (1-2\eps)|A||B|$,  $|A|, |B|\geq (1-2\eps) n/2$, and also that $\delta^0(T[A]), \delta^0(T[B])\geq (1/6-2\eps)n$, since at most $t$ vertices were moved out of $A, B$. Hence, $V(T)=A\cup B\cup X$ is a $2\eps$-good partition.

Let $C_A=\{v\in X: |N^+(v)\cap A|\geq k\}$ and $C_B=\{v\in X: |N^-(v)\cap B|\geq k\}$. Then, all vertices of $C_A\cap C_B$ are precisely the $k$-connectors in $X$. Thus, if $|C_A\cap C_B|\geq t$ (or $|C_A\cap C_B|\geq t+1$ if $n-t\equiv 1\bmod 4$), then we are done by part \textit{(a)} of Lemma~\ref{lemma:using connectors}. 

Thus, in what follows we will assume that the largest matching of edges directed from $B$ to $A$ has size at most $k$, since otherwise by part \textit{(b)} of Lemma~\ref{lemma:using connectors} we have \[\Pb[T_p\text{ is Hamiltonian}]\geq 1-(1-p^2)^k-\sigma/2\geq 1-\left(\frac{\sigma}{t+1}\right)^2-\sigma/2\geq 1-\sigma,\]
where we have used the definition of $k=2\log_{\frac{1}{1-p^2}}((t+1)\sigma^{-1})$.

\medskip

Let us count the incoming edges to the set $A$. We have $\sum_{a\in A}d^{-} (a)\geq \delta^0(T) |A|$. Since every edge contained in $A$ is counted exactly once in this sum, we have
\[\binom{|A|}{2}+e(B, A)+e(X, A)= \sum_{a\in A}d^{-}(a)\geq \delta^0(T) |A|.\]

The edge counts $e(B, A)$ and $e(X, A)$ can be bounded as follows. Observe that no vertex $b\in B$ has more than $k+t$ outgoing edges to $A$, since otherwise it would have been moved to $X$. Similarly, no vertex of $A$ has more than $k+t$ incoming edges from $B$. Moreover, the edges directed from $B$ to $A$ do not contain a matching of size $k$, and therefore by K\"onig's theorem there exits a set of $k$ vertices which hits all edges directed from $B$ to $A$. Thus, there are at most $k(k+t)$ edges directed from $B$ to $A$, $e(B, A)\leq k(k+t)$.

Let us now show that no vertex $v\in C_A\backslash C_B$ has at least $n/6$ outgoing edges to $A$. Suppose such a vertex $v$ existed - since $v\notin C_B$, then $v$ was not moved from $A_0\cup B_0$ to $X_0$, and so we have $v\in X_0$. Since the partition $V(T)=X_0\cup A_0\cup B_0$ was chosen so that $|X_0|$ is minimal, the vertex $v$ could not be added to $A_0$ and keep the partition $\eps$-good, which means that $|N^-(v)\cap A_0|< n/6$. Since $|X|\leq 4\eps n$, we have \[|N^-(v)\cap B|\geq d^-(v)-|N^{-}(v)\cap A|-|X|\geq \left\lfloor \frac{n+2}{4}\right\rfloor-\frac{n}{6}-4\eps n\geq \frac{n}{100}.\]
But then the vertex $v$ is a $k$-connector, since it has at least $k$ in-neighbours in $B$ and $k$ out-neighbours in $A$, contradicting the assumption that $v\notin C_B$. Finally, note that vertices outside $C_A$ have fewer than $k$ outneighbours in $A$, by definition, and so $e(X, A)\leq |A|\cdot |C_A\cap C_B|+n/6\cdot |C_A\backslash C_B| + k|X\backslash C_A|\leq |A|\cdot |C_A\cap C_B|+n/6\cdot|C_A\backslash C_B|+4k\eps n$. 

Putting it all together, we find that
\[\binom{|A|}{2}+k(k+t)+|A|\cdot |C_A\cap C_B|+\frac{n}{6}|C_A\backslash C_B|+4k\eps n\geq \delta^0(T) |A|.\]
For simplicity, let us denote $s=|C_A\cap C_B|$, and rearrange the above inequality
\begin{align*}
k(k+t)+\frac{n}{6}|C_A\backslash C_B|+4k\eps n&\geq |A|\left[\delta^0(T)-s-\frac{|A|-1}{2}\right]\\
&\geq \frac{(1-2\eps)n}{2}\left[\delta^0(T)-s-\frac{|A|-1}{2}\right].
\end{align*}
Also, observe that we can write $k(k+t)+4k\eps n\leq 5k\eps n$ if $n$ is large enough. Thus, if we divide the inequality by $n/6$, we arrive at 
\[30k\eps+|C_A\backslash C_B|\geq \frac{6(1-2\eps)}{2}\left[\delta^0(T)-s-\frac{|A|-1}{2}\right]\geq 2\left[\delta^0(T)-s-\frac{|A|-1}{2}\right].
\]
Hence,
\[30k\eps+|C_A\backslash C_B|\geq 2\left[\delta^0(T)-s-\frac{|A|-1}{2}\right].\]
Since $30k\eps\leq \frac{1}{3}$ and all other values in the inequality are integers, we conclude that the inequality still holds even when the term $30\eps k$ is removed. If we count the outgoing edges from $B$, we can perform a completely symmetric argument to arrive at the inequality
\[|C_B\backslash C_A|\geq 2\left[\delta^0(T)-s-\frac{|B|-1}{2}\right].\]
Let us write $|X|=x$ and observe that $|C_A\backslash C_B|+|C_B\backslash C_A|=|C_A\cup C_B|-|C_A\cap C_B|\leq x-s$, since $C_A\cup C_B\subseteq X$. Thus, adding the above two inequalities and using $|A|+|B|=n-x$, we find
\[x-s\geq |C_A\backslash C_B|+|C_B\backslash C_A|\geq 2\left[2\delta^0(T)-2s-\frac{|A|+|B|-2}{2}\right]\geq 2\left[2\delta^0(T)-2s-\frac{n}{2}+\frac{x+2}{2}\right].\]
Rearranging, we get $3s\geq 4\delta^0(T) - n+2$. If $n-t\equiv 1\bmod 4$, we then have $\delta^0(T)\geq \frac{n-1-t}{4}+t$, and so $3s\geq 1+3t$, implying that $s\geq t+1$. Otherwise, we get $\delta^0(T)\geq \frac{n-4-t}{4}+t$, and so $3s\geq 3t-2$, implying that $s\geq t$. By applying part \textit{(a)} of Lemma~\ref{lemma:using connectors}, this concludes the proof.
\end{proof}

\section{Concluding remarks}

As we mentioned in the introduction, the original question of Erd\H{o}s and Faudree can be interpreted as a question of robustness of Hamiltonicity in Dirac graphs, under sampling random vertex subsets. However, there is no reason to limit the study of this type of question to Hamiltonicity only. It would also be interesting to study this notion of robustness in relation to the classical theorems which guarantee the existence of spanning structures under a minimum degree assumption. To be more precise, if a graph $G$ with minimum degree $\delta(G)\geq c n$ must always contain a spanning structure $H$, and if $S\subseteq V(G)$ is a random vertex subset, how likely is $S$ to contain the same spanning structure? 

In another direction, we saw that Theorem~\ref{thm:main} applies for any $p\in (0, 1)$ (and also if $p$ decreases sufficiently slowly as a function of $n$). However, the question of Erd\H{o}s and Faudree has only been addressed for $p=1/2$ in the graph case, and it seems that the arguments of \cite{DKM} do not readily generalize to all $p$. Therefore, if $G$ is an $(n+1)$-regular $2n$-vertex graph and $S\subseteq V(G)$ is sampled by including each vertex independently with probability $p$, it remains an interesting problem to determine the probability that $G[S]$ is Hamiltonian as $p$ varies.

\end{document}